\documentclass[12pt]{amsart}

\usepackage{latexsym,amssymb,amsmath,amsthm,amscd,graphicx, xcolor, enumerate}

\usepackage{times}
\usepackage{microtype}
\usepackage{setspace}
\usepackage[margin=1.2in]{geometry}

\usepackage{cite}

\usepackage[colorlinks=true, pdfstartview=FitV, linkcolor=blue,
citecolor=blue, urlcolor=blue]{hyperref}

\numberwithin{equation}{section}
\allowdisplaybreaks

\numberwithin{equation}{section}
\newtheorem{theorem}{Theorem}[section]
\newtheorem{lemma}[theorem]{Lemma}
\newtheorem{corollary}[theorem]{Corollary}
\newtheorem{proposition}[theorem]{Proposition}

\theoremstyle{definition}

\newtheorem{remark}[theorem]{Remark}
\newtheorem{definition}[theorem]{Definition}

\theoremstyle{remark}

\title[On Herz-Bochkarev limiting problem]{On Herz-Bochkarev limiting problem}

\author[E. Nursultanov]{Erlan Nursultanov}
\address{Erlan Nursultanov\\ Department of Mathematics and Informatics\\
Lomonosov Moscow State University, Kazakhstan Branch\\
and
Institute of Mathematics and Mathematical Modelling, Almaty
Kazakhstan}
\email{er-nurs@yandex.kz}

\author[A. Ghorbanalizadeh]{Arash Ghorbanalizadeh}

\address{Arash Ghorbanalizadeh\\
	Department of Mathematics\\
    Institute for Advanced Studies in Basic Sciences (IASBS) \\
    45137-66731 Zanjan, Iran
    \\ and 
    Department of Mathematics\\
 Nazarbayev University \\
	53 Kabanbay Batyr Ave, 010000  Astana\\
	Kazakhstan \\ 
    }
\email{ arash.ghorbanalizadeh@nu.edu.kz; ghorbanalizadeh@iasbs.ac.ir}

\author[D. Suragan]{Durvudkhan Suragan}

\address{Durvudkhan Suragan \\
	Department of Mathematics \\
 Nazarbayev University \\
	53 Kabanbay Batyr Ave, 010000 Astana  \\
	Kazakhstan}

\email{durvudkhan.suragan@nu.edu.kz}

\begin{document}

\begin{abstract}  
This paper studies Hausdorff–Young-type inequalities within the framework of Lorentz spaces $L_{p,q}$. Focusing on the dependence of the associated constants on the integrability parameter \(p\), we derive optimal bounds in the limiting case \( p \rightarrow 2 \), addressing the Herz-Bochkarev problem. The results obtained refine the pioneering estimates in \cite{Boch} and are comparable to recent advances in \cite{ST}. The main ingredients of our approach are new grand Lorentz space techniques.
\end{abstract}  

\maketitle

\section{Introduction and main results}

One of the classical problems in Fourier analysis is the following: Given a function, estimate its Fourier coefficients in terms of the function itself.  Parseval's equality is probably the best-known example in this field. Consider a function \( f \) that is integrable in \([0, 1]\). Let \( \{ \varphi_k \}_{k=1}^\infty \) denote an orthonormal system. The corresponding Fourier coefficients of \( f \) with respect to the system \( \{ \varphi_k \}_{k=1}^\infty \) are defined by
\begin{equation}\label{Coef}
a_k = a_k(f):= \int_0^1 f(x) \overline{\varphi_k(x)} \,dx.    
\end{equation}

If \( f \in L_2([0,1]) \), then Parseval's equality is \(\|f\|_{L_2} = \|a\|_{\ell_2}:= \left( \sum\limits_{k \in \mathbb{Z}} |a_k|^2 \right)^{1/2},\) where
\(
a_k = a_k(f) := \int_0^1 f(x) e^{-2\pi i k x} \, dx
\)
are the Fourier coefficients of \( f \) with respect to the trigonometric system. 

Historically, Young and Hausdorff extended this result to spaces \( L_p :=L_{p}([0, 1])\). They proved that 
\begin{theorem}\label{T-1-1}
Let \( 1 < p \leq 2 \) and \( f \in L_{p} \). 
Then, the following results hold:  
\begin{equation}\label{YH-1}
\|a\|_{\ell_{p'}} = \left( \sum_{k\in \mathbb{Z}} |a_k|^{p'} \right)^{1/p'}\leq \|f\|_{L_p}.
\end{equation}
\end{theorem}

As usual, here and afterwards, \(p'\) denotes the conjugate exponent to \(p\), that is,  \(p'=\frac{p}{p-1}\).
Note that this theorem does not hold for \( p > 2 \) (see \cite{Zyg}, Ch. 12).  
Theorem \ref{T-1-1} was first proved by Young for the case where \( p' \) is an even positive number. The key idea in Young's proof was the convolution inequality obtained by him. In 1923, Hausdorff extended Young's result to the whole range \( [1,2] \). In the same year, F. Riesz proved that an analogue of Theorem \ref{T-1-1} holds for any uniformly bounded orthonormal system.

In the context of Lorentz spaces 
$L_{p,q}:=L_{p,q}([0,1])$, various methods have been developed to establish inequalities of this kind. Among them, interpolation arguments play a central role; see~\cite{Hu}, \cite[p.~490]{Stein}, and \cite[Chap.~V]{Stein-1}. We recall that the Hausdorff--Young inequality in the Lorentz space setting takes the form
\begin{equation}
\|a\|_{\ell_{p',q}} \lesssim \|f\|_{L_{p,q}}, \quad  1 < p < 2,\; 1 \leq q \leq \infty.
\end{equation}
Zygmund’s maximal theorem is another fundamental tool in the analysis of Lorentz and Orlicz spaces, particularly in the study of the Hausdorff--Young inequality. This principle underlies the work in~\cite{CT-1, CT-2, CT-3}. Note that in~\cite{A}, the author studied Zygmund’s result for \( L_{p}(\mathbb{R}) \) spaces using Carleson’s theorem. In addition, a number of studies explore new function spaces in which these inequalities can be formulated and analyzed; see, for example,~\cite{BennettRudnick1983, Si-2, NS, NurSur25}. Another significant approach to understanding the behavior of Fourier coefficients in Lorentz spaces--or more generally, rearrangement-invariant spaces--is through weighted Lebesgue-type (or rearrangement-invariant) inequalities for the Fourier transform $\hat{f}$ on \( \mathbb{R}^n \). For instance, the following types of estimates:
\[
\begin{cases}
\|u\hat{f}\|_{L_q} \lesssim \|f \, v\|_{L_p}, & \text{if}\, u\,\text{and}\,\ v^{-1} \text{ are even and non-increasing weights}, \\
\|\hat{f}\|_{X} \lesssim \|f\|_Y, & \text{for rearrangement-invariant spaces } X\,\text{and}\, Y,
\end{cases}
\]
imply the corresponding operator norm inequalities:
\[
\begin{cases}
\|uT(f)\|_{L_q} \lesssim \|f \, v\|_{L_p},  \\
\|T(f)\|_{X} \lesssim \|f\|_Y, 
\end{cases}
\]
for any operator \( T \) satisfying the endpoint bounds:
\[
\|T(f)\|_{L_\infty} \lesssim \|f\|_{L_1} 
\quad \text{and} \quad 
\|T(f)\|_{L_2} \lesssim \|f\|_{L_2}.
\]

We refer to recent work~\cite{ST} for an application of this method to establish optimal Hausdorff--Young-type inequalities even in limiting cases. In fact, the authors address and resolve the Herz-Bochkarev problem. Thus, the following optimal estimate was established in \cite{ST}:
\[
\left( \sum_{n=1}^{\infty}  \frac{1}{n \log^{q/2}(n+1)} \left(\sum_{k=1}^{n} |a_k|^2 \right)^{q/2} \right)^{1/q} \lesssim \|f\|_{L_{2,q}}, \quad 2 < q \leq \infty.
\]

In the present paper, first we revisit the Hausdorff--Young inequality in Lorentz spaces, emphasizing how the constant in this optimal inequality behaves under variations of the parameter \( p \):
\begin{theorem}\label{T-e}
Let \( 1 < p < 2 \) and \( f \in L_{p,q} \). Then 
\begin{equation}\label{HYL}
\|a\|_{\ell_{p',q}} \leq  c \frac{1}{\left( \frac{1}{p} - \frac{1}{2} \right)^{\frac{1}{2}}} \|f\|_{L_{p,q}} , \quad 0 < q \le \infty,
\end{equation}
where \(c\) is independent of \(p\) and \(a= \{a_k\}_{k=1}^{\infty}\) are the corresponding Fourier coefficients of \( f \) with respect to an arbitrary given orthonormal system.
\end{theorem}
Inequality~\eqref{HYL} is an attempt to introduce a scalar formulation that confirms Bochkarev's counterexample from \cite{Boch} in the limiting case \( p \rightarrow 2 \). Note that this limiting case problem was first studied by Herz in \cite{Herz}. For detailed discussions, we refer to \cite{DoVer} and \cite{ST}.

To derive the limiting case results, the main ingredient of our argument is the construction of the so-called grand Lorentz space and is denoted by \(G^{\theta}L_{p,q}:=G^{\theta}L_{p,q}([0,1])\), see Definition \ref{grand Lorentz space}, which enlarges the Lorentz space with respect to the parameter \( p \) which seems an effective framework for recovering improved inequalities at critical endpoints.  We prove the following theorems in the setting of grand Lorentz spaces. 

\begin{theorem}\label{0.3}
Let \( 2 \le q \le \infty \), \( \theta \ge 0 \), \(\theta_1=\theta+\frac{1}{2}\) and \(f \in G^{\theta}L_{2,q}\). Then  
\[
\|a\|_{G^{\theta_1}\ell^{*}_{2,q}} \lesssim \|f\|_{G^{\theta}L_{2,q}}.
\]
In particular,  
\begin{equation}
\|a\|_{G^{\frac{1}{2}}\ell^{*}_{2,q}} \lesssim \|f\|_{L_{2,q}}.
\end{equation}
Here, \(a= \{a_k\}_{k=1}^{\infty}\) are the corresponding Fourier coefficients of \( f \) with respect to an arbitrary given orthonormal system and the grand sequence spaces $G^{\theta_{1}}\ell^{*}_{2,q}$ are defined by \eqref{grandseqspace}.
\end{theorem}

 \begin{theorem} \label{T-2-2}
Let \( 0 < q < 2 \), \(  \theta \ge 0\), \( \theta_1 = \theta + \frac{1}{q} \) and $f\in G^{\theta}L_{2,q}$. Then  
\[
\|a\|_{G^{\theta_1}\ell^{*}_{2,q}} \lesssim \|f\|_{G^{\theta}L_{2,q}}.
\]

Here, \(a= \{a_k\}_{k=1}^{\infty}\) are the corresponding Fourier coefficients of \( f \) with respect to an arbitrary given orthonormal system and  the grand sequence spaces $G^{\theta_{1}}\ell^{*}_{2,q}$ are defined by \eqref{grandseqspace}.
\end{theorem}
Theorem \ref{0.3} refines Bochkarev's result. In \cite{Boch, Boch-1}, Bochkarev established the following lower bound for functions \( f \in L_{2,q}([0,1]) \):
\[
\frac{1}{ (\log~ (n+1))^{\frac{1}{2}-\frac{1}{q}}}  \left(\sum_{k=1}^n (a^*_k)^2 \right)^{\frac{1}{2}} \lesssim \|f\|_{L_{2,q}}, \quad 2 < q \le \infty,
\]
where \( \{a_k^*\}_{k=1}^{\infty} \) is the non-increasing rearrangement of the sequence \( \{|a_k|\}_{k=1}^{\infty} \).
This estimate is sharp in the sense that the exponent in the logarithmic term is optimal: it cannot be improved. More precisely, for any \( \epsilon > 0 \), there exists a function \( f \in L_{2,q} \) such that the corresponding sum with the logarithmic factor replaced by \( (\log (n+1))^{\frac{1}{2} - \frac{1}{q} - \epsilon} \) diverges. That is, any attempt to strengthen the inequality by reducing the logarithmic power will necessarily exclude some functions in \( L_{2,q} \). In contrast, our lower bound, in  Theorem \ref{0.3}, satisfies the following inequality:
\begin{equation*}
 \sup\limits_{n}\frac{1}{ (\log~( n+1))^{\frac{1}{2}-\frac{1}{q}}}  \left(\sum_{k=1}^n (a^*_k)^2 \right)^{\frac{1}{2}} \lesssim \|a\|_{G^{\frac{1}{2}}\ell^*_{2,q}}, \quad 2 < q \le \infty,
\end{equation*}
which shows that the norm of the Fourier coefficients in the grand Lorentz space \( G^{\frac{1}{2}}\ell^*_{2,q} \) dominates the  expression appearing in Bochkarev's lower bound, implying that our result provides an improved lower sharp bound in this critical case, see Remark \ref{ORBB}.

Moreover, we further refine these limiting case inequalities combining the grandization and interpolation techniques. To this end, we establish the following interpolation theorem for grand Lorentz sequence spaces.

\begin{theorem}\label{interpolation}
Let \(1\le p < q_0 < q_1 \le \infty\) be such that 
\[
\frac{1}{q}=\frac{1-\eta}{q_0} + \frac{\eta}{q_1}, \qquad 0<\eta <1.
\]
Then,
\begin{equation}\label{Inter}
(G^{\frac{1}{p}}\ell^{*}_{p,q_0} , G^{\frac{1}{p}}\ell^{*}_{p,q_1} )_{\eta \tau} \hookrightarrow \Lambda_{p,q,\tau},  \quad 0<\tau \le \infty.
\end{equation}

\end{theorem} 

Here, $\Lambda_{p,q,\tau}$ are the sequence spaces given in \eqref{sequencespace} and the grand sequence spaces $G^{\theta}\ell^{*}_{p,q}$ are defined by \eqref{grandseqspace}.

As our final result, we employ the preceding interpolation theorem to derive the following statement, which, in a particular case, yields a version of the Herz–Bochkarev inequality featuring the best known sharp lower bound to the best of our knowledge (cf. \cite{ST}). 
\begin{theorem}\label{Best-LB}
 Let \(2<q  \le \infty\) and \(f \in L_{2,q,\tau}\). Then
 \[
 \|a\|_{\Lambda_{2,q,\tau}} \lesssim \|f\|_{L_{2,q, \tau}},  \quad 0<\tau \le \infty, 
 \]
 where \(a= \{a_k\}_{k=1}^{\infty}\) are the corresponding Fourier coefficients of \( f \) with respect to an arbitrary given orthonormal system and $L_{2,q, \tau}$ is defined by \eqref{Lpqtau}. In particular case, when \(\tau=q\), we have 
\begin{equation}
\left(\sum_{k=0}^{\infty} \left(\sup_{s \ge k } \left( \frac{1}{s} \sum_{m=1}^{ 2^s } (a_m^*)^2 \right)^{\frac{1}{2}} \right)^{q} \right)^{\frac{1}{q}}  \lesssim \|f\|_{L_{2,q}}.
\end{equation}

\end{theorem}

In particular, this lower bound is comparable to the lower bound recently established in \cite{ST}, and we have 
\[
 \sum_{k=1}^{\infty}  \frac{1}{k \log^{q/2}(k+1)} \left(\sum_{m=1}^{k} (a^{*}_m)^2  \right)^{q/2}\lesssim \sum_{k=0}^{\infty} \left(\sup_{s \ge k } \left( \frac{1}{s} \sum_{m=1}^{ 2^s } (a_m^*)^2 \right)^{\frac{1}{2}} \right)^{q}  , \quad 2 < q \leq \infty,
\]
 see Remark \ref{Sharp lower}.

This paper has a simple structure. 
In Section \ref{pre}, we introduce grand Lorentz spaces, as well as other necessary tools and notation. 
The proofs of our main results are given in Section \ref{proofs}.

\section{Preliminaries}
\label{pre}

Let \( f \) be a complex-valued measurable function in \([0,1]\), and let \( \mu \)   denote the Lebesgue measure in \([0,1]\). The distribution function \(D_{f}(\lambda)\) is given by 
\[
D_{f}(\lambda) = \mu (\{ x \in [0,1] ;\,\, |f(x)| > \lambda \}), \qquad \lambda \ge 0.
\]
\(D_{f}(\lambda)\) is non-increasing and right-continuous, and the non-increasing rearrangement \(f^{*}\) of \(f\) is defined for \(t > 0\) by 
\[
f^*(t) = \inf \{ \lambda ; \,\, D_{f}(\lambda) \le t\}.
\]
We adopt the convention that $\inf \emptyset = \infty$, which implies that $f^*(t) = \infty$ whenever $D_f > t$, for all $\lambda \geq 0$. Note that $f^*$ is decreasing and supported in $[0, 1]$.
\begin{definition}
For \( 0 < p \leq \infty \), the Lorentz quasi-norm of \( f \), denoted by \( \|f\|_{L_{p,q}} \), is defined as follows:
\[
\|f\|_{L_{p,q}} =
\begin{cases}
\left( \int_0^1 \left( t^{1/p} f^*(t) \right)^q \, \frac{dt}{t} \right)^{1/q} & \text{if } q < \infty, \\
\sup\limits_{0 < t \leq 1} \, t^{1/p} f^*(t) & \text{if } q = \infty.
\end{cases}
\]
The set of all measurable functions \( f \) for which \( \|f\|_{L_{p,q}} < \infty \) is denoted by \( L_{p,q}([0,1]) \) and is called a Lorentz space with indices \( p \) and \( q \).

Similarly, for a complex-valued sequence \( a = \{a_k\}_{k=1}^\infty \), we define the Lorentz quasi-norm as
\[
\|a\|_{\ell_{p,q}} =
\begin{cases}
\left( \sum\limits_{k=1}^\infty \left( k^{1/p}\, a_k^* \right)^q \, \frac{1}{k} \right)^{1/q} & \text{if } 0< q < \infty, \\
\sup\limits_{n \geq 1} \, n^{1/p} \, a_k^*, & \text{if } q = \infty,
\end{cases}
\]
where \( \{a_k^*\}_{k=1}^\infty \) is the non-increasing rearrangement of the sequence \( \{|a_k|\}_{k=1}^\infty \). The space of all sequences  such that \( \|a\|_{\ell_{p,q}} < \infty \) is denoted by \( \ell_{p,q} \)--the Lorentz sequence space.
\end{definition}

\begin{definition}\label{grand Lorentz space}
Let \( 0 < p \leq \infty \). The \emph{grand Lorentz space} \( G^{\theta}L_{p,q}([0,1]) \) consists of all measurable functions \( f \) for which the following quasi-norm is finite (see \cite{NHS}):
\[
\|f\|_{G^{\theta}L_{p,q}([0,1])} =
\begin{cases}
\displaystyle\sup_{0<\varepsilon<1} \, \varepsilon^{\theta} 
\left( \int_0^1 \left( t^{\frac{1}{p} + \varepsilon} f^*(t) \right)^q \frac{dt}{t} \right)^{\frac{1}{q}} 
& \text{if } \theta \geq 0,\; 0< q < \infty,\, \\[10pt]
\displaystyle\sup_{0<\varepsilon<1} \, \varepsilon^{\theta} 
\sup_{0 < t \leq 1} \, t^{\frac{1}{p} + \varepsilon} f^*(t) 
& \text{if } \theta \geq 0,\; q = \infty.
\end{cases}
\]

Similarly, for \( \theta \geq 0 \), \( 0 < p \leq \infty \), and \( 0 < q \leq \infty \), the \emph{grand Lorentz sequence space} \( G^{\theta}\ell_{p,q} \) is defined by:
\[
\| a \|_{G^{\theta}\ell_{p,q}} = \sup_{0<\varepsilon<1} \, \varepsilon^{\theta} \left( \sum_{k=1}^\infty \left(k^{\frac{1}{p} - \varepsilon} a_k^*\right)^q \,   \frac{1}{k}\right)^{1/q}, \; 0<q<\infty,
\]
and
\[
\| a \|_{G^{\theta}\ell_{p,\infty}} = \sup_{0<\varepsilon<1} \sup_{k \geq 1} \, \varepsilon^{\theta} \, k^{\frac{1}{p} - \varepsilon} \, a_k^* 
\]
for $q = \infty$.
\end{definition}

We note that the grand Lorentz spaces are monotone concerning the second index $q$, that is, if \( 0 < q < q_1 \leq \infty \), then  
\[
G^{\theta}L_{p,q}\subseteq G^{\theta}L_{p,q_1} \quad \text{and} \quad \|f\|_{G^{\theta}L_{p,q_1}} \leq \|f\|_{G^{\theta}L_{p,q}},
\]
see \cite{NHS}.

Note that replacing \( a_k^* \) with  \( \left( \frac{1}{k} \sum\limits_{m=1}^k (a_m^*)^{\alpha} \right)^{\frac{1}{\alpha}}, \, \alpha \geq 1, \) yields an equivalent norm for the grand Lorentz sequence space (respectively, for the Lorentz sequence space). These spaces are denoted by \( G^{\theta}\ell^{*}_{p,q} \), and \( \ell^{*}_{p,q} \), respectively. In this paper, we will establish our results for sequences within the framework of the equivalent quasi-norm defined for the case \( \alpha = 2 \) as 
\begin{equation} \label{grandseqspace}
\| a \|_{G^{\theta}\ell^{*}_{p,q}} := \begin{cases}
\sup\limits_{0<\varepsilon<1} \varepsilon^{\theta} \left( \sum\limits_{k=1}^\infty \left( k^{\frac{1}{p}-\varepsilon} \left( \frac{1}{k} \sum\limits_{m=1}^{k} (a_m^*)^p \right)^{1/2} \right)^q \frac{1}{k} \right)^{1/q} 
& \text{if } \theta \geq 0,\; 0< q < \infty,\, \\[10pt]
\sup\limits_{0<\varepsilon<1} \sup\limits_{k \ge 1} \, \varepsilon^{\theta} \, k^{\frac{1}{p}-\varepsilon} \left( \frac{1}{k} \sum\limits_{m=1}^{k} (a_m^*)^p \right)^{1/2} 
& \text{if } \theta \geq 0,\; q = \infty.
\end{cases}
\end{equation}

In what follows, we use the following Hardy inequalities. 

\begin{proposition}\label{Hardy-1}
Let $\alpha > 0$, $0 < r \leq q <\infty$, then
\begin{equation}\label{Hardy-1}
\left( \int_0^1 \left(  t^{-\alpha} \left( \int_0^t (f^*(s))^r \, ds \right)^{\frac{1}{r}} \right)^q \frac{dt}{t} \right)^{\frac{1}{q}}
\leq \frac{1}{(r\alpha)^{\frac{1}{r}}} \left(\int_0^1 (t^{\frac{1}{r} - \alpha} f^*(t))^q \frac{dt}{t} \right)^{\frac{1}{q}}.
\end{equation}
\end{proposition}
\begin{proof}Using Minkowski's inequality followed by a change of variables, we obtain:
\begin{align*}
\left( \int_0^1 \left(  t^{-\alpha} \left( \int_0^t (f^*(s))^r \, ds \right)^{\frac{1}{r}} \right)^q \frac{dt}{t} \right)^{\frac{1}{q}} & = \left( \int_0^1 \left(  t^{-\alpha} \left( \int_0^1 (f^*(st))^r \,t ds \right)^{\frac{1}{r}} \right)^q \frac{dt}{t} \right)^{\frac{1}{q}}
\\
&=\left( \int_0^1   \left( \int_0^1 t^{-r \alpha} (f^*(st))^r \,t ds \right)^{\frac{q}{r}}\frac{dt}{t} \right)^{{\frac{r}{q}} \frac{1}{r}}
\\
&\le \left( \int_0^1   \left(\int_0^1 t^{-q \alpha} (f^*(st))^q \,t^{\frac{q}{r}} \frac{dt}{t} \right)^{\frac{r}{q}} ds\right)^{ \frac{1}{r}}
\\
&=\left( \int_0^1  s^{r \alpha -1}ds \right)^{ \frac{1}{r}} \left(\int_0^1 u^{-q \alpha} (f^*(u))^q \,u^{\frac{q}{r}} \frac{du}{u} \right)^{\frac{1}{q}} 
\\
&=\left( \frac{1}{r \alpha} \right)^{ \frac{1}{r}} \left(\int_0^1 ( u ^{\frac{1}{r}- \alpha}f^*(u))^q  \frac{du}{u}\right)^{\frac{1}{q}}.
\end{align*}
\end{proof}

\begin{corollary}
Let \( 0 < r \leq q <\infty\), \( \theta \ge 0 \), and \( \theta_1 = \theta + \frac{1}{r} \). Then,
\[
\inf_{0<\varepsilon \leq \frac{1}{r}} 
\varepsilon^{-\theta}
\left( 
\int_0^1 t^{-\varepsilon} 
\left( \int_0^t (f^*(s))^r \, ds \right)^{\frac{1}{r}} 
dt \right)^{\frac{1}{q}}
\leq \| f \|_{G^{-\theta_1}L_{r,q}}.
\]
\end{corollary}

\begin{proposition}\label{Hardy-Se}
Let \( \alpha > 0 \), \( 0 < r \leq q  \,\le \, \infty \). Then,
\begin{equation}\label{Hardy-2}
 \left(\int_{0}^{1} 
\left(t^{\alpha} 
\left( \int_{t}^{1} (f^*(s))^r ds \right)^{\frac{1}{r}}\right)^q
\frac{dt}{t}\right)^{\frac{1}{q}}
\leq 
\frac{1}{(r\alpha)^{\frac{1}{r}}} 
\left(\int_{0}^{1} 
\left( t^{\frac{1}{r} + \alpha} f^*(t) \right)^q 
\frac{dt}{t}\right)^{\frac{1}{q}}.
\end{equation}
\end{proposition}

\begin{proof}
We start with case \(q< \infty\). Applying a change of variables, using Minkowski's inequality, and then performing another change of variables, we get 
\begin{align*}
  \left(\int_{0}^{1} 
\left(t^{\alpha} 
\left( \int_{t}^{1} (f^*(s))^r ds \right)^{\frac{1}{r}}\right)^q
\frac{dt}{t}\right)^{\frac{1}{q}} & = \left(\int_{0}^{1} 
\left(t^{\alpha} 
\left( \int_{1}^{\frac{1}{t}} (f^*(ts))^r tds \right)^{\frac{1}{r}}\right)^q
\frac{dt}{t}\right)^{\frac{1}{q}}
\\
&= \left(\int_{0}^{1} 
\left( t^{r\alpha} \int_{1}^{\frac{1}{t}} (f^*(ts))^r tds \right)^{\frac{q}{r}}
\frac{dt}{t}\right)^{\frac{r}{q}\times \frac{1}{r}}
\\
&\le \left( \int_{1}^{\infty} \left(\int_{0}^{1} t^{\alpha q } (f^{*}(ts))^{q} t^{\frac{q}{r}} \frac{dt}{t} \right)^{\frac{r}{q}} ds\right)^{\frac{1}{r}}
\\
&=\left( \int_{1}^{\infty} s^{-1-r\alpha}  ds\right)^{\frac{1}{r}}  \left(\int_{0}^{1}   (u^{\alpha+ \frac{1}{r} }f^{*}(u))^{q}  \frac{du}{u} \right)^{\frac{1}{q}}
\\
&=\left( \frac{1}{r \alpha}\right)^{\frac{1}{r}}  \left(\int_{0}^{1}   (u^{\alpha+ \frac{1}{r} }f^{*}(u))^{q}  \frac{du}{u} \right)^{\frac{1}{q}}.
\end{align*}

For $q=\infty$, we show 
\[
\sup_{t > 0} \left[ t^{\alpha} 
\left( \int_{t}^{1} (f^*(s))^r \, ds \right)^{\frac{1}{r}} \right]
\leq 
\frac{1}{(r\alpha)^{\frac{1}{r}}} 
\sup_{t > 0} \left[ t^{\frac{1}{r} + \alpha} f^*(t) \right].
\]

Using a change of variable, one has
\begin{align*}
\sup_{t > 0} \left[ t^{\alpha} 
\left( \int_{t}^{1} (f^*(s))^r \, ds \right)^{\frac{1}{r}} \right] &= \sup_{t > 0} t^{\alpha} \left( \int_{1}^{\frac{1}{t}} (f^*(tu))^r \, t du \right)^{\frac{1}{r}} 
\\
& \le \left( \int_{1}^{\infty}  \sup_{t > 0}  ( t^{\alpha + \frac{1}{r}}  f^*(tu) )^r \,  du  \right)^{\frac{1}{r}} 
\\
& \le \left( \int_{1}^{\infty}  \sup_{\zeta > 0}  ( (\frac{\zeta}{u})^{\alpha + \frac{1}{r}}  f^*(\zeta) )^r \,  du  \right)^{\frac{1}{r}} 
\\
& \le \left( \int_{1}^{\infty}  \frac{1}{u^{1+ r\,\alpha}}  \,  du  \right)^{\frac{1}{r}}  \sup_{\zeta > 0}   \zeta^{\alpha + \frac{1}{r}}  f^*(\zeta) 
\\
& \le \left(  \frac{1}{\alpha r} \right)^{\frac{1}{r}}  \sup_{\zeta > 0}   \zeta^{\alpha + \frac{1}{r}}  f^*(\zeta) .
\end{align*}
This completes the proof of the desired result.
\end{proof}

\begin{corollary}
Let \( 0  \le r \leq q  \,< \infty  \), \(\theta \ge 0 \), and \( \theta_1 = \theta + \frac{1}{r} \). Then  
\[
\sup_{0<\varepsilon<\frac{1}{r}} \varepsilon^{\theta_1}  
\left( \int_0^{1} \left(t^{\varepsilon}  
\left( \int_t^1 (f^*(s))^r \, ds \right)^{\frac{1}{r}} \right)^{q} 
\frac{dt}{t} \right)^{\frac{1}{q}}  
\leq \|f\|_{G^{\theta}L_{r,q}}.
\]
\end{corollary}

\section{Proofs of main results}
\label{proofs}

\subsection{Proof of Theorem \ref{T-e}.} We split the statement of Theorem \ref{T-e} into two parts depending on the range of the parameter $q$ and prove separately the following lemmas. 

\begin{lemma}\label{lemma}
 Let \( 1 < p < 2 \), and let \( f \in L_{p,q} \). Then  
\begin{equation}\label{Main}
\|a\|_{\ell^{*}_{p',q}} \leq  c \frac{1}{\left( \frac{1}{p} - \frac{1}{2} \right)^{\frac{1}{2}}} \|f\|_{L_{p,q}} , \quad 2 \le q \le \infty,
\end{equation}
where the constant $c$ is independent of \(p\) and \(a= \{a_k\}_{k=1}^{\infty}\) are the corresponding Fourier coefficients of \( f \) with respect to  an arbitrary given orthonormal system.   
\end{lemma}
\begin{proof}[Proof of Lemma \ref{lemma}]
Let us recall that 
\[
\|f\|_{L_2}= \left(\int_{0}^{1} (f^{*}(t))^2 dt\right)^{\frac{1}{2}} \asymp \left(\sum_{k=0}^{\infty} \left(2^{-\frac{k}{2}}f^{*}(2^k)\right)^2\right)^{\frac{1}{2}}. 
\]
Let $k \in \mathbb{Z}_{+}$ and  $\tau=2^{-k}$. We decompose the function \( f \in L_{2}([0,1]) \)  into two functions \( f_1(x) \) and \( f_2(x) \) as follows:  

\begin{equation}\label{9}
f_0(x) =
\begin{cases} 
    f(x)-f^{*}(\tau) & \text{if } ~|f(x)| \ge  f^{*}(\tau), \\ 
    0 & \text{if} ~~ |f(x)| <  f^{*}(\tau),
\end{cases}    
\end{equation}

\begin{equation}\label{10}
f_1(x) = f(x) - f_0(x).
\end{equation}

Now we calculate the left-hand side of \eqref{Main} as follows:

\begin{align*}
 \|a\|_{\ell_{p',q}^{*}}  
 &\asymp \left( \sum_{k=0}^{\infty} \left( 2^{\frac{k}{p'}} \left( \frac{1}{2^k} \sum_{m=0}^{2^{k}} (a_{m}^{*}(f))^{2} \right)^{1/2} \right)^{q} \right)^{\frac{1}{q}} 
 \\
 &\lesssim \left( \sum_{k=0}^{\infty} \left( 2^{\frac{k}{p'}} \left( \frac{1}{2^k} \sum_{m=0}^{2^{k}} \left( a_{m}^{*}(f_1) + a_{m}^{*}(f_0) \right)^{2} \right)^{1/2} \right)^{q} \right)^{\frac{1}{q}} 
 \\
 &\lesssim \left( \sum_{k=0}^{\infty} \left( 2^{\frac{k}{p'}} \left( \frac{1}{2^k} \sum_{m=0}^{2^{k}} (a_{m}^{*}(f_1))^{2} \right)^{1/2} \right)^{q} \right)^{\frac{1}{q}} \\
 &\quad + \left( \sum_{k=0}^{\infty} \left( 2^{\frac{k}{p'}} \left( \frac{1}{2^k} \sum_{m=0}^{2^{k}} (a_{m}^{*}(f_0))^{2} \right)^{1/2} \right)^{q} \right)^{\frac{1}{q}} 
 =: I_1 + I_2.
\end{align*}

First, consider $I_2$. We have \(|a_{m}(f)| \le \|f\|_{L_1}\), that is, 
\begin{align*}
 I_2=    \left( \sum_{k=0}^{\infty} \left( 2^{\frac{k}{p'}} \left( \frac{1}{2^k} \sum_{m=0}^{2^{k}} (a_{m}^{*}(f_0))^{2} \right)^{1/2} \right)^{q} \right)^{\frac{1}{q}} \lesssim \left( \sum_{k=0}^{\infty} \left(2^{\frac{k}{p'}} \|f_0\|_{L_1}\right)^{q}\right)^{\frac{1}{q}}.
\end{align*}
We also have  
\begin{align*}
 \|f_0\|_{L_1} = \int_{0}^{1} f^*_0(t) dt =  \int_{0}^{\tau} f^*(t)-f^*(\tau) dt \le \int_{0}^{2^{-k}}f^*(t) dt \asymp \sum_{m=k}^{\infty} 2^{-m} f^{*}(2^{-m}).
\end{align*}
Combining these with Hardy's inequality we get
\begin{align*}
I_2 &\le \left( \sum_{k=0}^{\infty} \left(2^{\frac{k}{p'}} \|f_0\|_{L_1}\right)^{q}\right)^{\frac{1}{q}} 
\\
&=  \left( \sum_{k=0}^{\infty} \left(2^{\frac{k}{p'}} \sum_{m=k}^{\infty} 2^{-m} f^{*}(2^{-m})\right)^{q}\right)^{\frac{1}{q}} 
\\
&\asymp  \left(\sum_{m=0}^{\infty} \left( 2^{\frac{-m}{p}} f^{*} (2^{-m})\right)^{q} \right)^{\frac{1}{q}}
\\
&\asymp \|f\|_{L_{p,q}}. 
\end{align*}

Next, consider $I_1$. By using the fact that for a convex function \( h \), we have  
\[
h \left(\frac{1}{n} \sum\limits_{m=1}^{n}  a_{m}^{*}  \right) \le \frac{1}{n} \sum\limits_{m=1}^{n}  h\left(a_{m}^{*}\right),
\]  
and applying Parseval’s inequality, we obtain

\begin{align*}
 I_1&=  \left( \sum_{k=0}^{\infty} \left(2^{k\left(\frac{1}{p'}-\frac{1}{2}\right)}  ~ \left(\sum_{m=0}^{2^{k}} (a_{m}^{*}(f_1))^2\right)^{\frac{1}{2}}\right)^{q}\right)^{\frac{1}{q}}
\\
&\le \left( \sum_{k=0}^{\infty} \left(2^{k\left(\frac{1}{2}-\frac{1}{p}\right)}  ~ \|f_1\|_{L_2}\right)^{q}\right)^{\frac{1}{q}}.
\end{align*}

On the other hand, by using \eqref{9} we have
\[
f_1(x)=
\begin{cases}
f^*(\tau) & \text{if } |f(x)|\ge f^*(\tau),\\[1mm]
f(x) & \text{if } |f(x)|< f^*(\tau).
\end{cases}
\]
Let
\[
E=\{x\in[0,1]: \,|f(x)|>f^*(\tau)\},
\]
then 
\[
\mu(E) = \mu(\{t : \, f^{*}(t) > f^{*}(\tau)\}) .
\]
Thus, the decreasing rearrangement of \(f_1\) is
\[
f_1^*(t)=
\begin{cases}
f^*(2^{-k}), & 0\le t< 2^{-k},\\[1mm]
f^*(t), & 2^{-k}\le t\le 1.
\end{cases}
\]
Therefore, 
\begin{align}\label{l2est}
\|f_1\|_{L_2} = \left(\int_0^1 (f_1^*(t))^2\,dt\right)^{\frac{1}{2}}& = \left(\int_0^{2^{-k}} (f^*(2^{-k}))^2\,dt + \int_{2^{-k}}^{1}  (f^*(t))^2\right)^{\frac{1}{2}} \notag
\\
&= 2^{\frac{-k}{2}} (f^*(2^{-k})) + \left(\int_{2^{-k}}^{1}  (f^*(t))^2\right)^{\frac{1}{2}}
\\
&\asymp
2^{\frac{-k}{2}} (f^*(2^{-k})) +  \left(\sum_{m=1}^{k} 2^{-m}(f^{*}(2^{-m}))^2\right)^{\frac{1}{2}} , \notag
\end{align}
and  
\begin{align*}
 I_1&\le \left( \sum_{k=0}^{\infty} \left(2^{k\left(\frac{1}{2}-\frac{1}{p}\right)} \left( 2^{\frac{-k}{2}} (f^*(2^{-k})\right) \right)^q \right)^{\frac{1}{q}}\\
 & + \left( \sum_{k=0}^{\infty} \left(2^{k\left(\frac{1}{2}-\frac{1}{p}\right)}  ~ \left(\sum_{m=1}^{k} 2^{-m}(f^{*}(2^{-m}))^2\right)^{\frac{1}{2}}\right)^{q}\right)^{\frac{1}{q}}
 \\
 &= J_1 +J_2.
\end{align*}

Here, we have
\begin{align*}
J_1= \left( \sum_{k=0}^{\infty} \left(2^{-\frac{k}{p}}  ~ f^*(2^{-k})\right)^{q}\right)^{\frac{1}{q}} \lesssim \|f\|_{L_{p,q}}.
\end{align*}

Finally, to estimate \(J_2\), we apply the dyadic form of Hardy's inequality  \eqref{Hardy-2}, 
\begin{align*}
&\left( \sum_{k=0}^{\infty} \left(2^{k\left(\frac{1}{2}-\frac{1}{p}\right)}  ~ \left(\sum_{m=0}^{k} 2^{-m}(f^{*}(2^{-m}))^2\right)^{\frac{1}{2}}\right)^{q}\right)^{\frac{1}{q}}
\\
&\qquad \qquad \qquad \qquad \qquad \qquad \lesssim \frac{1}{\left(\frac{1}{p}-\frac{1}{2}\right)^{\frac{1}{2}}}
\left( \sum_{k=0}^{\infty} \left( 2^{-\frac{k}{p}} f^*(2^{-k}) \right)^q \right)^{\frac{1}{q}}
\\
&\qquad \qquad \qquad \qquad \qquad\qquad\lesssim \frac{1}{\left(\frac{1}{p}-\frac{1}{2}\right)^{\frac{1}{2}}}
\|f\|_{L_{p,q}}.
\end{align*}

\end{proof}

Now we focus on the range \( 0 < q < 2 \).

\begin{lemma}\label{lemma2}
Let \( 1 < p < 2 \) and \( f \in L_{p,q} \). Then  
\begin{equation}\label{Main-2}
\|a\|_{\ell^{*}_{p',q}} \leq  c \frac{1}{\left( \frac{1}{p} - \frac{1}{2} \right)^{\frac{1}{q}}} \|f\|_{L_{p,q}} , \quad 0< q < 2,
\end{equation}
where the constant $c$ is independent of \(p\) and \(a= \{a_k\}_{k=1}^{\infty}\) are the corresponding Fourier coefficients of \( f \) with respect to an arbitrary given orthonormal system.
\end{lemma}
\begin{proof}[Proof of Lemma \ref{lemma2}]
As in the proof of Lemma \ref{lemma}, let \( f \in L_{2}([0,1]) \) decompose into two functions \( f_1(x) \) and \( f_2(x) \) as follows:  

\begin{equation}\label{9-2}
f_0(x) =
\begin{cases} 
    f(x)-f^{*}(\tau), & \text{if } ~|f(x)| \ge  f^{*}(\tau) \\ 
    0, & \text{if} ~~ |f(x)| <  f^{*}(\tau)
\end{cases}    
\end{equation}

\begin{equation}\label{10-2}
f_1(x) = f(x) - f_0(x).
\end{equation}

Now we estimate the left-hand side of \eqref{Main-2} as follows:

\begin{align*}
 \|a\|_{\ell^{*}_{p'q}} &\lesssim \left( \sum_{k=0}^{\infty} \left( 2^{\frac{k}{p'}} \left( \frac{1}{2^k} \sum_{m=0}^{2^{k}} (a_{m}^{*}(f_1))^{2} \right)^{1/2} \right)^{q} \right)^{\frac{1}{q}} \\
 & + \left( \sum_{k=0}^{\infty} \left( 2^{\frac{k}{p'}} \left( \frac{1}{2^k} \sum_{m=0}^{2^{k}} (a_{m}^{*}(f_0))^{2} \right)^{1/2} \right)^{q} \right)^{\frac{1}{q}}  
 \\
 &= I_1 +I_2.
\end{align*}

We have 
\begin{align*}
 I_2 \lesssim \left( \sum_{k=0}^{\infty} \left(2^{\frac{k}{p'}} \|f_0\|_{L_1}\right)^{q}\right)^{\frac{1}{q}} &=  \left( \sum_{k=0}^{\infty} \left(2^{\frac{k}{p'}} \sum_{m=k}^{\infty} 2^{-m} f^{*}(2^{-m})\right)^{q}\right)^{\frac{1}{q}}. 
\end{align*}
One can see that estimating the term corresponding to \(f_0\), namely \(I_2\) in Lemma \ref{lemma} works for range \(1\le q<2\), so we need to consider case $0<q<1$. By using Jensen's inequality and the fact $2^{-qm} \ge 2^{m}$, we have 
\begin{align*}
 I_2 &\le    \left( \sum_{k=0}^{\infty} 2^{\frac{kq}{p'}} \sum_{m=k}^{\infty} 2^{-qm} \left(f^{*}\right(2^{-m}))^q\right)^{\frac{1}{q}} 
 \\
 &=  \left( \sum_{m=0}^{\infty}  2^{-qm} \left(f^{*}\right(2^{-m}))^q \sum_{k=0}^{m} 2^{\frac{kq}{p'}} \right)^{\frac{1}{q}} 
 \\
 &\asymp_{q} \left( \sum_{m=0}^{\infty}  \left(2^{\frac{-m}{p}} f^{*}(2^{-m})\right)^q  \right)^{\frac{1}{q}} \asymp \|f\|_{L_{p,q}}.
\end{align*}

Following a computation similar to that in the proof of Lemma \ref{lemma} and using Jensen's inequality, for \(0<q <2\), we get
\begin{align*}
 I_1&\le \left( \sum_{k=0}^{\infty} \left(2^{k\left(\frac{1}{2}-\frac{1}{p}\right)}  ~ \|f_1\|_{L_2}\right)^{q}\right)^{\frac{1}{q}}
 \\
 &\le \left( \sum_{k=0}^{\infty} \left(2^{k\left(\frac{1}{2}-\frac{1}{p}\right)} \left( 2^{\frac{-k}{2}} (f^*(2^{-k})\right) \right)^q \right)^{\frac{1}{q}} \\
 & + \left( \sum_{k=0}^{\infty} \left(2^{k\left(\frac{1}{2}-\frac{1}{p}\right)}  ~ \left(\sum_{m=0}^{k} \left(2^{\frac{-m}{2}}f^{*}(2^{-m})\right)^2\right)^{\frac{1}{2}}\right)^{q}\right)^{\frac{1}{q}}
 \\
 & \lesssim \|f\|_{L_{p,q}} + \frac{1}{\left(\frac{1}{p}-\frac{1}{2}\right)^{\frac{1}{q}}} \|f\|_{L_{p,q}}.
\end{align*}
\end{proof}

\begin{proof}[Proof of Theorem \ref{T-e}]
Combining Lemmas~\ref{lemma} and \ref{lemma2}, we obtain the desired result.    
\end{proof}

\subsection{Proofs of Theorem \ref{0.3} and Theorem \ref{T-2-2}}

\begin{proof}[Proof of Theorem \ref{0.3}]
Let \( \varepsilon > 0 \). By Lemma \ref{lemma} with $\frac{1}{p}=\frac{1}{2}+\varepsilon$, we have 
\[
\left( \sum_{k=1}^\infty \left( k^{\frac{1}{2} - \varepsilon} \left( \frac{1}{k} \sum_{m=1}^{k} (a_m^*)^{2} \right)^{1/2} \right)^q \frac{1}{k} \right)^{1/q} \leq c \varepsilon^{-\frac{1}{2}}  
\left( \int_0^1 (t^{\frac{1}{2} + \varepsilon} f^*(t))^q \frac{dt}{t} \right)^{\frac{1}{q}}.
\]

Now multiplying both sides by \(\varepsilon^{\theta_1} \), we get  
\[
\varepsilon^{\theta_1}  
\left( \sum_{k=1}^\infty \left( k^{ - \varepsilon} \left(  \sum_{m=1}^{k} (a_m^*)^{2} \right)^{1/2} \right)^q \frac{1}{k} \right)^{1/q} 
\leq c\varepsilon^{\theta}  
\left( \int_0^1 (t^{\frac{1}{2} + \varepsilon} f^*(t))^q \frac{dt}{t} \right)^{\frac{1}{q}}  
\leq c \|f\|_{GL^{\theta}_{2,q}}.
\]

Thus, 
\[
\|a\|_{G^{\theta_1}\ell^{*}_{2,q}} \lesssim \|f\|_{G^{\theta}L_{2,q}}.
\]

\end{proof}

\begin{remark}\label{ORBB}
Let us show that the lower bound in Theorem \ref{0.3} refines Bochkarev's lower bound.
A direct computation gives 
\begin{align*}
\|a\|_{G^{\frac{1}{2}}\ell^{*}_{2,q}} 
&= \sup_{0 < \varepsilon < 1} \varepsilon^{\frac{1}{2}} \left( \sum_{k=1}^\infty \left( k^{-\varepsilon} \left( \sum_{m=1}^{k} (a_m^*)^2 \right)^{1/2} \right)^q \frac{1}{k} \right)^{1/q} \\
&\ge \sup_{0 < \varepsilon < 1} \varepsilon^{\frac{1}{2}} \left( \sum_{k=n}^\infty \left( k^{-\varepsilon} \left( \sum_{m=1}^{k} (a_m^*)^2 \right)^{1/2} \right)^q \frac{1}{k} \right)^{1/q} \\
&\ge \sup_{0 < \varepsilon < 1} \varepsilon^{\frac{1}{2}} \left( \sum_{k=n}^\infty k^{-q\varepsilon - 1} \right)^{1/q} \left( \sum_{m=1}^{n} (a_m^*)^2 \right)^{1/2} \\
&\ge \sup_{0 < \varepsilon < 1} \varepsilon^{\frac{1}{2}}  \frac{(n+1)^{-\varepsilon}}{\varepsilon^{1/q}} \left( \sum_{m=1}^{n} (a_m^*)^2 \right)^{1/2} \\
&= \frac{(n+1)^{-\varepsilon}}{\varepsilon^{\frac{1}{q} - \frac{1}{2}}} \left( \sum_{m=1}^{n} (a_m^*)^2 \right)^{1/2} .
\end{align*}
Let \( \varepsilon = \frac{1}{\ln(n+1)} \). Then it follows that
\begin{align*}
\|a\|_{G^{\frac{1}{2}}\ell^{*}_{2,q}} &\ge  \frac{(n+1)^{\frac{-1}{ln(n+1)}}}{\left( \frac{1}{\ln(n+1)} \right)^{\frac{1}{q} - \frac{1}{2}}} \left( \sum_{m=1}^{n} (a_m^*)^2 \right)^{1/2} \\
&= \frac{e^{-1}}{\left( \frac{1}{\ln(n+1)} \right)^{\frac{1}{q} - \frac{1}{2}}} \left( \sum_{m=1}^{n} (a_m^*)^2 \right)^{1/2} \\
&= \frac{e^{-1}}{\left( \ln(n+1) \right)^{\frac{1}{2} - \frac{1}{q}}} \left( \sum_{m=1}^{n} (a_m^*)^2 \right)^{1/2}.
\end{align*}
\end{remark}

\begin{proof}[Proof of Theorem \ref{T-2-2}]
Let \( \varepsilon > 0 \). From Lemma \ref{lemma2} with $\frac{1}{p}=\frac{1}{q}+\varepsilon$, it follows that  
\[
\left( \sum_{k=1}^\infty \left( k^{\frac{1}{2} - \varepsilon} \left( \frac{1}{k} \sum_{m=1}^{k} (a_m^*)^{2} \right)^{1/2} \right)^q \frac{1}{k} \right)^{1/q} 
\leq c \varepsilon^{-\frac{1}{q}}  
\left( \int_0^1 (t^{\frac{1}{2} + \varepsilon} f^*(t))^q \frac{dt}{t} \right)^{\frac{1}{q}}.
\]
Multiplying both sides by \(\varepsilon^{\theta_1} \), we have  
\[
\varepsilon^{\theta_1}  
\left( \sum_{k=1}^\infty \left( k^{ - \varepsilon} \left(  \sum_{m=1}^{k} (a_m^*)^{2} \right)^{1/2} \right)^q \frac{1}{k} \right)^{1/q}  
\leq c \varepsilon^{\theta}  
\left( \int_0^1 (t^{\frac{1}{2} + \varepsilon} f^*(t))^q \frac{dt}{t} \right)^{\frac{1}{q}}  
\leq c \|f\|_{G^{\theta}L_{2,q}}.
\]

It yields  
\[
\|a\|_{G^{\theta_1}\ell^{*}_{2,q}} \lesssim \|f\|_{G^{\theta}L_{2,q}}.
\]

\end{proof}

\subsection{Proofs of Theorem \ref{interpolation} and Theorem \ref{Best-LB}.}

By applying the same argument as in the Remark \ref{ORBB}, we obtain the following lemma.

\begin{lemma}\label{ImBoch}
Let \(1 \le p < q < \infty\). Then
\begin{equation}
\sup_{s > 0} s^{\frac{1}{q}} \left( \frac{1}{s} \sum_{m=1}^{ 2^s } (a_m^*)^p \right)^{\frac{1}{p}} \lesssim \|a\|_{G^{\frac{1}{p}}\ell^{*}_{p,q}},
\end{equation}
where
\[
\|a\|_{G^{\theta}\ell^{*}_{p,q}} = \sup_{0 < \varepsilon < 1} \varepsilon^{\theta} \left( \sum_{k=1}^\infty \left( k^{-\varepsilon} \left( \sum_{m=1}^k (a_m^*)^p \right)^{\frac{1}{p}} \right)^q \frac{1}{k} \right)^{\frac{1}{q}}
\]
and \(a_m^*\) is the non-increasing rearrangement sequence.
\end{lemma}

\begin{proof}[Proof of Lemma \ref{ImBoch}]
For \(n = 2^s \),  we have
\[
s^{\frac{1}{q}} \left( \frac{1}{s} \sum_{m=1}^n (a_m^*)^p \right)^{\frac{1}{p}} = s^{\frac{1}{q} - \frac{1}{p}} \left( \sum_{m=1}^n (a_m^*)^p \right)^{\frac{1}{p}}.
\]
Consider the norm restricted to \(k \geq n\):
\[
  \sup_{0 < \varepsilon < 1} \varepsilon^{\frac{1}{p}} \left( \sum_{k=n}^\infty \left( k^{-\varepsilon} \left( \sum_{m=1}^k (a_m^*)^p \right)^{\frac{1}{p}} \right)^q \frac{1}{k} \right)^{\frac{1}{q}} \leq \|a\|_{G^{\frac{1}{p}}\ell^{*}_{p,q}}.
\]
Since \(\sum\limits_{m=1}^k (a_m^*)^p \geq \sum\limits_{m=1}^n (a_m^*)^p\) for \(k \geq n\), we get
\[
 \sup_{0 < \varepsilon < 1} \varepsilon^{\frac{1}{p}} \left( \sum_{k=n}^\infty k^{-q\varepsilon - 1} \right)^{\frac{1}{q}} \left( \sum_{m=1}^n (a_m^*)^p \right)^{\frac{1}{p}} \leq \|a\|_{G^{\frac{1}{p}}\ell^{*}_{p,q}} .
\]
The integral test gives \(\sum\limits_{k=n}^\infty k^{-q\varepsilon - 1} \ge \frac{n^{-q\varepsilon}}{q \, \varepsilon}\), that is, we obtain
\[
 \frac{n^{-\varepsilon}}{\varepsilon^{\frac{1}{q}}} \lesssim \left( \sum_{k=n}^\infty k^{-q\varepsilon - 1} \right)^{\frac{1}{q}},
\]
and
\[
 \sup_{0 < \varepsilon < 1} \frac{\varepsilon^{\frac{1}{p} - \frac{1}{q}}}{n^{\varepsilon}} \left( \sum_{m=1}^n (a_m^*)^p \right)^{\frac{1}{p}} \lesssim \|a\|_{G^{\frac{1}{p}}\ell^{*}_{p,q}} . 
\]
Let us maximize the expression \(\varepsilon^{\frac{1}{p} - \frac{1}{q}} n^{-\varepsilon}\).  Set \(\Phi(\varepsilon)= \varepsilon^{\frac{1}{p} - \frac{1}{q}} n^{-\varepsilon}\) and
\[
g(\varepsilon)=\ln \Phi(\varepsilon) = \left(\frac{1}{p} - \frac{1}{q}\right) \ln \varepsilon - \varepsilon \ln n .
\]
If we differentiate and set the derivative to zero, we get 
\[
\Phi'(\varepsilon) = \Phi(\varepsilon) g'(\varepsilon)=\Phi(\varepsilon) \left[ \frac{(\frac{1}{p} - \frac{1}{q}) }{\varepsilon} - \ln n \right] =0, \quad  \Rightarrow \quad \varepsilon = \frac{\frac{1}{p} - \frac{1}{q}}{\ln n}.
\]
Substituting, we have 
\[
e^{-\left(\frac{1}{p} - \frac{1}{q}\right)} \left( \frac{1}{p} - \frac{1}{q} \right)^{\frac{1}{p} - \frac{1}{q}} (\ln n)^{\frac{1}{q} - \frac{1}{p}} \left( \sum_{m=1}^n (a_m^*)^p \right)^{\frac{1}{p}} \lesssim \|a\|_{G^{\frac{1}{p}}\ell^{*}_{p,q}}.
\]
The relation \(\ln n \asymp s\) gives
\[
s^{\frac{1}{q} - \frac{1}{p}} \left( \sum_{m=1}^n (a_m^*)^p \right)^{\frac{1}{p}} \lesssim  e^{\frac{1}{p} - \frac{1}{q}} \left(\frac{1}{p}- \frac{1}{q} \right)^{\frac{1}{q} - \frac{1}{p}} \|a\|_{G^{\frac{1}{p}}\ell^{*}_{p,q}}.
\]
Hence,
\[
\sup\limits_{s > 0} s^{\frac{1}{q}} \left( \frac{1}{s} \sum\limits_{m=1}^{ 2^s } (a_m^*)^p \right)^{\frac{1}{p}} \lesssim \|a\|_{G^{\frac{1}{p}}\ell^{*}_{p,q}}.
\]
\end{proof}

Let us recall the definition of the sequence space
\begin{equation}\label{sequencespace}
\Lambda_{p, q, \tau} = \left\{ a = \{a_k\}_{k=1}^\infty : \left( \sum_{k=1}^{\infty} \left( k^{\frac{1}{q}} \sup_{m \geq k} \left( \frac{1}{m} \sum_{l=1}^{2^{m}} (a^*_l)^p \right)^{\frac{1}{p}} \right)^{\tau} \frac{1}{k} \right)^{\frac{1}{\tau}} < \infty \right\},
\end{equation}
where \( \{a^*_k\}_{k=1}^\infty \) denotes the non-increasing rearrangement of the sequence \( \{a_k\}_{k=1}^\infty \). Here, the parameters \( p, q,\) and $\tau$ are assumed to be in \( (0, \infty) \).

\begin{proof}[Proof of Theorem \ref{interpolation}]
Let \( a \in (G^{\frac{1}{p}}\ell^{*}_{p,q_0}, G^{\frac{1}{p}}\ell^{*}_{p,q_1})_{\eta \, \tau} \), and let \( a = a_0 + a_1 \) be an arbitrary decomposition with \( a_0 \in G^{\frac{1}{p}}\ell^{*}_{p,q_0} \) and \( a_1 \in G^{\frac{1}{p}}\ell^{*}_{p,q_1} \). For any \(k \in \mathbb{N} \) we have 
\begin{align*}
k^{\frac{1}{q}} \sup_{s\ge k} &\left( \frac{1}{s} \sum_{m=0}^{2^{s}} (a^{*}_m)^p\right)^{\frac{1}{p}} 
\\
& \le k^{\frac{1}{q}-\frac{1}{q_0}} \sup_{s\ge k} \left( k^{\frac{1}{q_0}}\left(\frac{1}{s} \sum_{m=0}^{2^{s}} (a^{*}_{0 \,m})^p\right)^{\frac{1}{p}} + k^{\frac{1}{q_0}-\frac{1}{q_1}} k^{\frac{1}{q_1}} \left(\frac{1}{s} \sum_{m=0}^{2^{s}} (a^{*}_{1\,m})^p\right)^{\frac{1}{p}}\right)
\\
& \le k^{\frac{1}{q}-\frac{1}{q_0}} \left(  \sup_{s\ge k} s^{\frac{1}{q_0}}\left(\frac{1}{s} \sum_{m=0}^{2^{s}} (a^{*}_{0\,m})^p\right)^{\frac{1}{p}} + k^{\frac{1}{q_0}-\frac{1}{q_1}}  \sup_{s\ge k} s^{\frac{1}{q_1}} \left(\frac{1}{s} \sum_{m=0}^{2^{s}} (a^{*}_{1\,m})^p\right)^{\frac{1}{p}}\right).
\end{align*}
By applying Lemma \ref{ImBoch}, we get 
\begin{align*}
k^{\frac{1}{q}} \sup_{s\ge k} \left( \frac{1}{s} \sum_{m=0}^{2^{s}} (a^{*}_m)^p \right)^{\frac{1}{p}}  & \le k^{\frac{1}{q}-\frac{1}{q_0}} \left(  \|a_0\|_{G^{\frac{1}{p}}\ell^{*}_{p,q_0}}+ k^{\frac{1}{q_0}-\frac{1}{q_1}}  \|a_1\|_{G^{\frac{1}{p}}\ell^{*}_{p,q_1}}\right).
\end{align*}

Since this relation holds for an arbitrary decomposition, it follows that 
\begin{align}\label{gor}
\begin{split}
k^{\frac{1}{q}} \sup_{s\ge k} \left( \frac{1}{s} \sum_{m=0}^{2^{s}} (a^{*}_m)^p\right)^{\frac{1}{p}}  & \le 2^{ k\,\left[\frac{1}{q}-\frac{1}{q_0}\right]} \inf_{a=a_0 + a_1}  \left( \|a_0\|_{G^{\frac{1}{p}}\ell^{*}_{p,q_0}} + 2^{ k\, \left[\frac{1}{q_0}-\frac{1}{q_1}\right]} \|a_1\|_{G^{\frac{1}{p}}\ell^{*}_{p,q_1}}  \right) 
\\
&\le 2^{ k{\left[\frac{1}{q}-\frac{1}{q_0}\right]}} K\left( 2^ { k \, \left[\frac{1}{q_0}-\frac{1}{q_1}\right]} , a ; G^{\frac{1}{p}}\ell^{*}_{p,q_0} , G^{\frac{1}{p}}\ell^{*}_{p,q_1}\right)
\\
&= 2^{-\eta k{\left[\frac{1}{q_0}-\frac{1}{q_1} \right] }} 
K\left(2^{ k\, \left[\frac{1}{q_0}-\frac{1}{q_1}\right]} , a ; G^{\frac{1}{p}}\ell^{*}_{p,q_0} , G^{\frac{1}{p}}\ell^{*}_{p,q_1} \right).
\end{split}
\end{align}

For \(b = 2^{\frac{1}{q_0} - \frac{1}{q_1}} > 1\), we have 
\begin{align*}
 &\left(\int_{0}^{\infty} \left( t^{-\eta} K(t,a; G^{\frac{1}{p}}\ell^{*}_{p,q_0} , G^{\frac{1}{p}}\ell^{*}_{p,q_1})\right)^{\tau} \frac{dt}{t}\right)^{\frac{1}{\tau}}   
 \\
 &\asymp  \left(\sum_{k=0}^{\infty} \left( b^{-\eta k} K(b^{k},a; G^{\frac{1}{p}}\ell^{*}_{p,q_0} , G^{\frac{1}{p}}\ell^{*}_{p,q_1})\right)^{\tau} \right)^{\frac{1}{\tau}} 
 \\
 & = \left(\sum_{k=0}^{\infty} \left( 2^{-\eta k \left[\frac{1}{q_0} - \frac{1}{q_1}\right]} K(2^{k \left[\frac{1}{q_0} - \frac{1}{q_1}\right]},a; G^{\frac{1}{p}}\ell^{*}_{p,q_0} , G^{\frac{1}{p}}\ell^{*}_{p,q_1})\right)^{\tau} \right)^{\frac{1}{\tau}}
 \\
 & \ge_{\eqref{gor}} \left(\sum_{k=0}^{\infty} \left( 2^{\frac{k}{q}} \sup_{s \ge 2^{k} } \left( \frac{1}{s} \sum_{m=1}^{ 2^s } (a_m^*)^p \right)^{\frac{1}{p}} \right)^{\tau} \right)^{\frac{1}{\tau}}
 \\
 & \asymp \left(\sum_{k=0}^{\infty} \left( k^{\frac{1}{q}} \sup_{s \ge k } \left( \frac{1}{s} \sum_{m=1}^{ 2^s } (a_m^*)^p \right)^{\frac{1}{p}} \right)^{\tau} \frac{1}{k}\right)^{\frac{1}{\tau}}
 \\
& = \|a\|_{\Lambda_{p,q,\tau}}.
\end{align*}
\end{proof}

We recall the space $L_{p,q,\tau}([0,1])$, $0<p,q,\tau \leq \infty,$ defined as the space of measurable functions $f$ such that

\begin{equation}\label{Lpqtau}
\|f\|_{L_{p,q,\tau}} := \left( \sum_{k=1}^{\infty} \left( k^{\frac{1}{q}} \xi_k^{*}(f) \right)^\tau \right)^{1/\tau} < \infty. \end{equation}

Here, $\left\{ \xi_k^{*}(f) \right\}_{k=1}^{\infty}$ is the non-increasing rearrangement of the sequence 

\[ \xi_m(f) = \left( \int_{2^{-m-1}}^{2^{-m}} (f^*(t))^p dt \right)^{1/p}, \quad m \in \mathbb{N}. \]

It gives a generalization of the space $L_{p,q}([0,1])$. Thus, if $ \tau = q$, then $L_{p,q,q}([0,1]) = L_{p,q}([0,1])$, and in the case $p = q = \tau$, we get the Lebesgue space.  

%$ L_{p,p,p}([0,1]) = L_{p}([0,1]) + c $.

This space was introduced by Nursultanov (see \cite{Nursultanov-book}). For these spaces, in particular, the following fact holds: for $ 0 < q_{0} < q_{1} \leq \infty $,

\[ \left( L_{p,q_0}, L_{p,q_1} \right)_{\eta\, \tau} = L_{p,q,\tau}([0,1]) \]
where
\[ \frac{1}{q} = \frac{1-\eta}{q_0} + \frac{\eta}{q_1}, \quad 0 < \eta < 1. \]

\begin{proof}[Proof of Theorem \ref{Best-LB}]
Let \(2<q<\infty\) and \(q_0, \, q_1\) such that \(2<q_0<q<q_1 \le \infty\). Let \( \{ a_m \} \) denote the sequence of Fourier coefficients of \( f \), as defined in \eqref{Coef}. Define the linear operator \( T \) by \(T(f) = \{ a_m \}_{m},\)
which maps \( f \) to the sequence of its Fourier coefficients. By Theorem \ref{0.3}, we have 
\begin{align*}
\begin{cases}
 \|Tf\|_{G^{\frac{1}{2}}\ell^{*}_{2,q_0}}  & \lesssim \|f\|_{L_{2,q_0}},
 \\
 \|Tf\|_{G^{\frac{1}{2}}\ell^{*}_{2,q_1}}   &\lesssim \|f\|_{L_{2,q_1}}.
\end{cases}    
\end{align*}
Now using Theorem \ref{interpolation} and \cite[Theorem 2]{Nur1997} (see \cite{Nursultanov-book} for more details), we have 
\[
(L_{2,q_0} , L_{2,q_1})_{\eta \tau} = L_{2,q,\tau} \qquad \frac{1}{q}= \frac{1-\eta}{q_0} + \frac{\eta}{q_1}.
\]
Thus, we arrive at
\[
\|Tf\|_{\Lambda_{2,q,\tau}} \lesssim \|f\|_{L_{2,q,\tau}}.
\]
\end{proof}

\begin{remark}\label{Sharp lower}
Let us compare the obtained new lower bound in Theorem \ref{Best-LB} to the lower bound
\begin{equation}\label{STLB}
\left( \sum_{n=1}^{\infty}  \frac{1}{n \log^{q/2}(n+1)} \left(\sum_{j=1}^{n} (a_j^*)^2 \right)^{q/2} \right)^{1/q}, \quad 2 < q \leq \infty,
\end{equation}

from the recent work \cite{ST}. 

We have 
\[
\sum_{n=1}^{\infty} = \sum_{k=0}^{\infty} \sum_{n=2^k}^{2^{k+1}-1}.
\]
Observe that for all \( n \in [2^k, 2^{k+1}) \), we have \( \log(n+1) \asymp \log(2^{k+1}) = (k+1) \log 2 \), and  \(n \asymp 2^k\) hence
\[
\frac{1}{n \log^{q/2}(n+1)} \asymp \frac{1}{2^k (k+1)^{q/2}}.
\]
Then, the sum becomes
\begin{align*}
\sum_{n=1}^{\infty} \frac{1}{n \log^{q/2}(n+1)} \left( \sum_{j=1}^{n} (a_j^*)^2 \right)^{\frac{q}{2}}
&\lesssim \sum_{k=0}^{\infty} \frac{1}{(k+1)^{q/2}} \left( \sum_{j=1}^{2^{k+1}} (a_j^*)^2 \right)^{\frac{q}{2}}
\\
&= \sum_{k=1}^{\infty} \frac{1}{k^{q/2}} \left( \sum_{j=1}^{2^{k}} (a_j^*)^2 \right)^{\frac{q}{2}}.
\end{align*}
Therefore, we obtain 
\begin{align*}
\left( \sum_{n=1}^{\infty} \frac{1}{n \log^{q/2}(n+1)} \left( \sum_{j=1}^{n} (a_j^*)^2 \right)^{\frac{q}{2}} \right)^{1/q}
&\asymp
\left(  \sum_{k=1}^{\infty} \frac{1}{k^{q/2}} \left( \sum_{j=1}^{2^{k}} (a_j^*)^2 \right)^{\frac{q}{2}}  \right)^{1/q}
\\
& \le \left(  \sum_{k=1}^{\infty}  \left( \sup_{s \ge k } \left( \frac{1}{s}\sum_{j=1}^{2^{s}} (a_j^*)^2 \right)^{\frac{1}{2}} \right)^{q}  \right)^{1/q}
\\
& \le \left(  \sum_{k=0}^{\infty}  \left(\sup_{s \ge k } \left( \frac{1}{s}\sum_{j=1}^{2^{s}} (a_j^*)^2 \right)^{\frac{1}{2}} \right)^{q}  \right)^{1/q}
\\
&\asymp \|a\|_{\Lambda_{2,q,q}}.
\end{align*}
Note that the lower bound \eqref{STLB} becomes
\[
\sup_{n} \frac{1}{\log^{\frac{1}{2}}(n+1) } \left( \sum_{j=1}^{n} (a_j^{*})^2\right)^{\frac{1}{2}}
\]
when \( q = \infty \).
\end{remark}

\section{Acknowledgements}

This research was funded by Nazarbayev University under
Collaborative Research Program Grant 20122022CRP1601.
Durvudkhan Suragan thanks Centre de Recerca Matemàtica for hosting a week research visit in May
2025, part of the Intensive Research Programme on Modern Trends in Fourier Analysis.

\end{document}